\documentclass[12pt,leqno]{amsart}
\usepackage{amssymb, amsmath}
\usepackage{mathtools}
\usepackage{graphics}
\usepackage{eepic}
\usepackage{epic}
\usepackage{bbm}
\usepackage{graphicx,color}
\usepackage{hyperref}

\newtheorem{theorem}{Theorem}[section]
\newtheorem{proposition}[theorem]{Proposition}
\newtheorem{corollary}[theorem]{Corollary}
\newtheorem{lemma}[theorem]{Lemma}

\theoremstyle{definition}

\theoremstyle{remark}
\newtheorem*{remark}{Remark}
\newcommand{\uT}{\text{\upshape{T}}}
\newcommand{\uP}{\mathcal{{P}}}
\newcommand{\dd}{\text{\upshape{d}}}
\def \E {\mathbb{E}}
\def \P {\mathbb{P}}
\def \N  {\mathbb{N}}
\def \R  {\mathbb{R}}

\def \Z  {\mathbb{Z}}

\newcounter{anc}
\renewcommand{\theanc}{\arabic{anc}}
\newcommand{\cst}{\refstepcounter{anc}\ensuremath{c_{\theanc}}}
\let\originalleft\left
\let\originalright\right
\renewcommand{\left}{\mathopen{}\mathclose\bgroup\originalleft}
\renewcommand{\right}{\aftergroup\egroup\originalright}
\title[New high-dimensional examples of ballistic RWRE]{
New high-dimensional examples of ballistic  random walks in random environment}

\author{Ryoki Fukushima}
\address{Research Institute for Mathematical Sciences, Kyoto University}
\email{ryoki@kurims.kyoto-u.ac.jp}

\author{Alejandro F.~Ram\'\i rez}
\address{Facultad de Matem\'aticas, Pontificia Universidad Cat\'olica de Chile}
\email{aramirez@mat.uc.cl}

\begin{document}
%

\maketitle

\begin{abstract}      
We give new criteria for ballistic behavior of random
walks in random environment which are perturbations of the simple symmetric
random walk on $\mathbb Z^d$ in dimensions $d\ge 4$.
Our results extend those of Sznitman~[Ann.~Probab.~31, no.~1, 285-322 (2003)] and the recent ones of Ram\'\i rez and Saglietti~[Preprint, arXiv:1808.01523], and
allow us to exhibit new examples in dimensions $d\ge 4$ 
of ballistic random walks which do not satisfy Kalikow's condition.  Our criteria implies ballisticity whenever the average of the local drift of the walk 
is not too small compared with an appropriate moment of the
centered environment. The proof relies on a concentration inequality of
Boucheron et al.~[Ann. Probab.~33, no.~2, 514-560 (2005)].
\end{abstract}

\section{Introduction}
A challenging open question about
the model of multidimensional random walk in random environment (RWRE), in the case
of an independent and identically distributed (i.i.d.)~uniformly elliptic environment, is to characterize 
when the random walk is ballistic or not in terms of the law of
the environment at a single site. 

The first result in this direction is~\cite{SZ99}, where it is proved that Kalikow's condition for the directional transience in~\cite{K81} implies ballisticity. Although the condition itself depends on the whole environment, one can derive some sufficient conditions in terms of the law of
the environment at a single site, such as~\cite[(2.36)]{SZ99} and~\cite[Proposition~5.1]{BS04}. 

Beyond Kalikow's condition,
partial answers to the question of ballisticity have been given for the environments which are small 
perturbations of the simple symmetric random walk on $\mathbb Z^d$. 
For dimensions $d\ge 3$, Sznitman~\cite{Sz03} 
derived conditions on the averaged local drift which imply ballisticity.
Essentially, he showed that whenever the averaged local drift in
a given direction is not too small with respect to the $L^\infty$-norm of the perturbation, 
one has ballistic behavior in that direction. As a corollary of his results,
he was able to provide examples of ballistic random walks
in dimensions $d\ge 3$ that do not satisfy
Kalikow's condition. 
An improvement of these results 
was obtained by Ram\'\i rez and Saglietti~\cite{RS18}. 
There, a ballisticity condition in terms of the $L^2$ and $L^\infty$-norms of the perturbation is proved, as well as a sharp sufficient condition for Kalikow's condition. The former condition is shown to improve that in~\cite{Sz03} when $d=3$. 
As a consequence, new examples of ballistic RWRE, which satisfy neither Kalikow's condition nor the conditions in~\cite{Sz03}, were obtained in dimension $d=3$. 

In this article we refine the results of Sznitman \cite{Sz03} in
dimensions $d\ge 4$,  providing a ballisticity condition in terms of $L^{2r}$ and $L^\infty$-norms of the perturbation. 
As a corollary, this gives new
examples of ballistic RWRE which satisfy neither Kalikow's condition nor the conditions in~\cite{Sz03}. 

For more details on earlier works, see Section~\ref{sec:background}.

\section{Background and main results}
\subsection{The model and basic assumptions}
Let $d\ge 2$ and write for each $x\in\mathbb Z^d$ its $\ell^1$-norm as $|x|$. Define $V:=\{x\in\mathbb Z^d:|x|=1\}$
and denote by $\mathcal P$ the set of all probability vectors
$p=(p(e))_{e\in V}$ on $V$. We now consider the environmental
space $\Omega:=\mathcal P^{\mathbb Z^d}$, and call each
element $\omega=(\omega(x))_{x\in\mathbb Z^d}\in\Omega$
an {\it environment}. For each $x\in\mathbb Z^d$, we denote the
components of $\omega(x)$ by $\omega(x,e)$ so that
$\omega(x)=(\omega(x,e))_{e\in V}$. The {\it random walk
in the environment $\omega$} starting from $x\in\mathbb Z^d$
is then defined as the Markov chain $X=(X_n)_{n\ge 0}$ which
starts from $x$, and has the transition probabilities
$$
P_{x,\omega}(X_{n+1}=y+e \mid X_n=y)=\omega(y,e)
$$
for all $y\in\mathbb Z^d$ and $e\in V$. We denote its law by
$P_{x,\omega}$ and call it the {\it quenched law} of the random walk. 
Let now $\mathbb P$ be a probability measure
defined on $\Omega$. Define the semi-direct product 
$P_x:=\mathbb P\otimes P_{x,\omega}$ on $\Omega\times(\mathbb{Z}^d)^{\mathbb N}$
as
$$
P_x(A\times B):=\int_A P_{x,\omega}(B)\dd \mathbb P.
$$
We call $P_x$ the {\it averaged} or {\it annealed} law of
the random walk starting from $x$.

Throughout this article, we will assume that $(\omega(x))_{x\in\mathbb Z^d}$
are i.i.d.~under $\mathbb P$. We will denote by $\mu$ the common law of
$\omega(x)$, $x\in\mathbb Z^d$, that is, $\mathbb P=\mu^{\Z^d}$. We will also assume that each
$\omega(x)$ is a small perturbation of the weights of the simple
symmetric random walk, that is,
\begin{equation}
\label{epsilon-def}
\epsilon=\epsilon(\mu):=4d\left\|
\omega(0)-\left(\frac{1}{2d},\ldots,\frac{1}{2d}\right)\right\|_{L^\infty(\mu)}\in (0,1),
\end{equation}
where for any random vector $v=(v(e))_{e\in V}$, we define its $L^\infty(\mu)$-norm as
$$
||v||_{L^\infty(\mu)}:=\inf\{m>0\colon |v(e)|\le m,\ \mu\text{-a.s.~for all } e\in V\}.
$$
Note that (\ref{epsilon-def}) implies that there exists an event
$\Omega_\epsilon$ with $\mathbb P(\Omega_\epsilon)=1$ such that
for any $\omega\in\Omega_\epsilon$, $x\in\mathbb Z^d$ and $e\in V$,
$$
\left|\omega(x,e)-\frac{1}{2d}\right|\le\frac{\epsilon}{4d}.
$$
As a consequence, it follows that $\mathbb P$ is {\it uniformly elliptic}, that is, 
$\omega(x,e)\ge\kappa$ for all $x\in\mathbb Z^d$ and $e\in V$ with
$$
\kappa:=\frac{1}{4d}.
$$

In the statement of our result, as well as earlier works, we use the {\it local drift of the RWRE at site $x$} defined by 
$$
d(x,\omega):=\sum_{e\in V}\omega(x,e)e,
$$
and the {\it average local drift in direction $e_1$} defined by
$$
\lambda:=\mathbb E[d(0)\cdot e_1].
$$ 
(We will omit $\omega$ when we write an expectation.)

\subsection{Background and earlier works}
\label{sec:background}
We explain the background on the question of ballisticity for RWRE which we address in this paper. 
Let us start by introducing the concepts of directional transience which is closely related to ballisticity. 
Given $l\in\mathbb S^{d-1}:=\{x\in\R^d\colon |x|=1\}$, one says that the
random walk is {\it transient in direction $l$} if
$$
\lim_{n\to\infty} X_n\cdot l=\infty\qquad P_0\text{-a.s.}
$$
On the other hand, one says that it is {\it ballistic in direction $l$}
if it satisfies
$$
\liminf_{n\to\infty}\frac{X_n\cdot l}{n}>0\qquad P_0\text{-a.s.}
$$
It is well-known that any  random walk in an i.i.d.~uniformly elliptic environment which is ballistic in direction $l$ satisfies
\begin{equation}
 \lim_{n\to\infty}\frac{X_n}{n}=v_\P,
\label{LLN}
\end{equation}
where $v_\P$ is deterministic and $v_\P\cdot l>0$. See, for example,~\cite[Theorem~12 and Remark~6]{DR14}. 
An open question about the RWRE model is whether in dimensions $d\ge 2$, 
every random walk in an i.i.d.~uniformly elliptic environment which is
transient in a given direction, is ballistic (see~\cite[p.243]{Sz04} for example). 

Historically, Kalikow had found a sufficient condition for the directional 
transience in \cite{K81} and that condition was later proved to imply 
ballisticity in~\cite{SZ99}, where it is called \emph{Kalikow's condition}.
In order to understand the relation between directional transience and ballisticity, several quantitative conditions for the directional transience have been introduced. 
In~\cite{Sz01, Sz02}, Sznitman defined the so-called conditions $(\uT)$ and $(\uT')$:
for $\gamma\in (0,1]$ and $l\in\mathbb S^{d-1}$, we say that condition
$(\uT)_\gamma|l$ is satisfied if there exists an open set $O\subset\mathbb S^{d-1}$ which contains $l$, such that for every $l'\in O$, 
$$
\limsup_{L\to\infty}\frac{1}{L^\gamma}\log P_0\left(X_{T_{U_{L,l'}}}\cdot l'<0\right)<0,
$$
where 
$$
U_{L,l'}:=\{x\in\mathbb Z^d: -L\le x\cdot l'\le L\}.
$$
Condition $(\uT)|l$ is nothing but $(\uT)_1|l$, while we say that condition $(\uT')|l$ is satisfied if $(\uT)_\gamma|l$ is satisfied for all $\gamma\in (0,1)$. 
In \cite{BDR14}, for each $M\ge 1$ and $l\in\mathbb S^{d-1}$,
the polynomial condition $(\uP)_M|l$ was introduced,
which is essentially defined as  the requirement that the probability
$P_0(X_{T_{U_{L,l'}}}\cdot l'<0)$ decays like $L^{-M}$ for
$l\in O$, where $O\subset\mathbb S^{d-1}$ contains $l$.
It is currently known that 
conditions $(\uP)_M|l$ for $M\ge 15d+5$, $(\uT)_\gamma|l$ for $\gamma\in (0,1)$,
$(\uT')|l$ and $(\uT)|l$ are all equivalent (see \cite{Sz02,BDR14,GR18}) and they imply~\eqref{LLN}. In this paper, we write $(\uP)$ for $(\uP)_{15d+5}|e_1$ to simplify the notation. 

Now we turn to the earlier works more directly related to our results. 
First, although it is natural to expect that the condition $\E[d(0)\cdot e_1]>0$ implies ballisticity in direction $e_1$, it is not the case since in~\cite{BSZ03}, examples of random walks with $\E[d(0)\cdot e_1]>0$ and $v_\P=0$ or even $v_\P\cdot e_1<0$ are given. This illustrates a complicated nature of the ballisticity question. 
On the other hand, it was shown in \cite{Sz03}, improving upon~\cite{Sz02}, that such a pathology does not occur when the environment is a small perturbation of simple symmetric random walk. 
More precisely, for any $\eta\in (0,1)$, there exists some 
$\epsilon_0=\epsilon(\eta,d)\in (0,1)$ such that if $\epsilon\le\epsilon_0$ (recall~\eqref{epsilon-def}) and
\begin{equation}
\label{e34}
\lambda:=\E[d(0)\cdot e_1]\ge
\begin{cases}
\epsilon^{2.5-\eta} &\qquad {\rm if}\quad d=3,\\
\epsilon^{3-\eta} &\qquad {\rm if}\quad d\ge 4,
\end{cases}
\end{equation}
then the random walk satisfies condition $(\uT')$ (and hence $(\uT)$) in direction $e_1$. Subsequently in \cite{RS18},
an improvement of this result was obtained comparing $\lambda$
to the variance of the environment at a given site.
To state this extension, 
we define $\tilde\delta(0,e):=\omega(0,e)-\mathbb E[\omega(0,e)]$ and
\begin{equation} 
 \sigma_{2r}:=\left(\sum_{e\in V}
\E\left[\tilde\delta(0,e)^{2r}\right]
  \right)^{\frac{1}{2r}} \le (2d)^{\frac{1}{2r}}\epsilon, 
\label{eq:sigma_def}
\end{equation}
where the inequality holds when $\P(\Omega_\epsilon)=1$. 
Note that by the H\"older inequality, we have $\sigma_2\le (2d)^{\frac{r-1}{r}}\sigma_{2r}$.
One of the main results in~\cite{RS18} says that if $d=3$, for any $\eta\in (0,1)$
there exists $\epsilon_0=\epsilon_0(\eta)\in (0,1)$ such that
if $\epsilon\le\epsilon_0$ and
$$
\lambda\ge\sigma_2\epsilon^{1.5-\eta},
$$
then condition $(\uP)$ is satisfied. 
It is natural to expect that
an improvement of Sznitman's result in~\cite{Sz03} should be also
possible in dimensions $d\ge 4$, replacing the right-hand
side of~\eqref{e34} by some quantity similar to the variance.
This is the content of our main result that we present in the next section.

\subsection{Main results}
\begin{theorem}
  \label{d4}
  Suppose $d\ge 4$ and $\P(\Omega_\epsilon)=1$. Then, for every $r\ge 144d^2$, there
  exists an $\epsilon_0(r)\in (0,1)$ such that if
  $\epsilon\le\epsilon_0$ and
\begin{equation}
\lambda:=\E[d(0\cdot e_1)]\ge\sqrt{r}\sigma_{2r}\epsilon^{2-\frac{1}{\sqrt{r}}},
\label{eq:assum}
\end{equation}
then condition $(\uT)$ holds in direction $e_1$. In particular,
$X$ is ballistic in direction $e_1$.
\end{theorem}


Theorem~\ref{d4} gives new examples (apart from those already given
in \cite{Sz03}) of ballistic random walks in dimensions $d\ge 4$
which satisfy $(\uT)$ but not Kalikow's condition.


\begin{corollary}
\label{coro1}
Given $\epsilon_0>0$, there exist RWRE in dimensions $d\ge 4$ satisfying~\eqref{eq:assum} and such that
\begin{enumerate}
\item[(i)] $\P(\Omega_\epsilon)=1$ for $\epsilon\le\epsilon_0$ and $\lambda:=\E[d(0)\cdot e_1]\le\epsilon^3$,
\item[(ii)] $(\uT)$ holds in direction $e_1$ but
Kalikow's condition fails in all directions.
\end{enumerate}
\end{corollary}


As in~\cite{Sz03} and~\cite{RS18}, such an example can be constructed
by first fixing $\rho\in (0,1]$ and by
setting $\mu$ to be the law of the random probability $\omega(0)$
on $\mathcal P$ given by
$$
\omega(0,e)=p(e)+\frac{\lambda}{2} e\cdot e_1 \quad \text{ for }e\in V,
$$
where $\lambda>0$ is a constant to be specified later and $(p(e))_{e\in V}$
is a random probability vector with distribution $\hat\mu$ on $\mathcal P$
which is isotropic (i.e., invariant under rotations of $\mathbb R^3$
which preserve $\mathbb Z^3$) and such that
$$
{\rm Var}_{\hat\mu}(p(e_1))-{\rm Cov}_{\hat\mu}(p(e_1),p(-e_1))\ge \rho\sigma_2(\hat\mu)>0.
$$
It is not difficult to check that given $r\ge 144 d^2$, one can choose
constants $k_1,k_2,k_3>0$ (depending only on $r$, $\epsilon_0$ and $\rho$)
such that
\begin{align*}
&\epsilon(\hat \mu)\le k_1\epsilon_0 \quad \textrm{and}\\
&k_2\epsilon (\hat\mu)^{1.9}\le \sigma_2\le (2d)^{\frac{r-1}{r}}\sigma_{2r}\le k_3\epsilon(\hat\mu)^{1.1}.
\end{align*}
We can now choose $\lambda$ so that both \eqref{eq:assum} and condition $(i)$ of Corollary \ref{coro1}
is satisfied. The fact that Kalikow's condition is not satisfied
is a consequence of~\cite[Theorem~5]{RS18}.

\section{Proof of Theorem \ref{d4}}
Let us first explain the outline of the proof, which largely follows the arguments in~\cite{Sz03} and~\cite{RS18}. 

In Section~\ref{green}, we introduce some notation and preliminary results. In particular, we quote a sufficient condition for $(\uP)$ from~\cite{RS18}, which we will verify in order to prove ballisticity. It is given in terms of an exit distribution from a large box. In~\cite{Sz03}, Sznitman developed a renormalization argument to reduce the problem to two conditions on $\hat\rho$ and $p$ defined below, which are of perturbative nature. We will not reproduce the argument but the result in Lemma~\ref{coro}. 

In Section~\ref{conclusion}, we show how to apply Lemma~\ref{coro} in our setting. One of the condition (on $p$) follows from relatively weak $L^\infty$-control on the perturbation. The other condition (on $\hat\rho$) is about a uniform positivity of the Green operator applied to the local drift in direction $e_1$, which we will denote by $G_U[d\cdot e_1](0)$. Assuming two propositions on this quantity, we complete the proof of ballisticity. 

In Section~\ref{lowergreen}, we prove the two propositions left unproved in the previous section. The first says that the mean value $\E[G_U[d\cdot e_1](0)]$ is positive and not too small, and the second says that $G_U[d\cdot e_1](0)$ is \emph{concentrated} around the mean. For the first one, we use the same bound as in~\cite{RS18}. For the second, we use a concentration inequality due to Boucheron, Bousquet, Lugosi and Massart \cite{BBLM05}, which is a high moment analogue of the Efron-Stein inequality (see also \cite{BLM13}). 
Previously, this concentration bound is proved by the Azuma--Hoeffding inequality in~\cite{Sz03}, and its variation that takes into account the variance in~\cite{RS18}. 

\subsection{Notation and preliminaries}
\label{green}
In this section, we introduce some notations and preliminary results which will be used in the proof. For $A\subset\Z^d$, we denote its outer boundary by 
\begin{equation}
 \partial A:=\{y\in\Z^d\setminus A\colon |y-z|=1 \text{ for some }z\in A\}
\end{equation}
and the first exit time of the random walk from $A$ by 
$T_A:=\inf\{n\ge 0\colon X_n \not\in A\}$.

In order to state a sufficient condition for $(\uP)$, we define a box and its frontal side for $M\in\N$ (to be fixed later) by 
\begin{align*}
 B&:=\left(-M, M\right)\times\left(-\frac14 M^3, \frac14 M^3\right)^{d-1},\\
\partial_+ B&:=\left\{x\in \partial B:x\cdot e_1\ge M\right\}.
\end{align*}
Following the argument in~\cite[(23) and (24)]{RS18}, we find that if 
\begin{equation}
P_0\left(X_{T_B}\not\in \partial_+ B\right)<\frac{1}{M^{15d+5}}
\text{ for some }M\ge \exp\{100+4d(\log \kappa)^2\},
\label{effective}
\end{equation} 
then condition $(\uP)$ holds. 

The condition~\eqref{effective} is a kind of \emph{effective} (i.e., finite volume) criterion but it is still hard to check in practice. We quote a result from~\cite{Sz03} that gives a bound on the left-hand side of~\eqref{effective} in terms of more explicit functional of $\omega$. To this end, note first that by using the quenched exit probability 
$$
q_B(\omega):=P_{0,\omega}(X_{T_B}\notin\partial_+ B)
\qquad {\rm and}\qquad \rho_B(\omega):=
\frac{q_B(\omega)}{1-q_B(\omega)},
$$
we can bound
\begin{equation}
  \label{sqrtrho}
P_0(X_{T_B}\notin\partial_+ B)=\mathbb E[q_B]\le
\mathbb E\left[\sqrt{\rho_B}\right].
\end{equation}
To estimate the right-hand side of (\ref{sqrtrho}), we are going to use~\cite[Theorem~1.1]{Sz03} which requires further notation. 
Let us first introduce the second \emph{mesoscopic} scale $L\in\N$ 
and define the slab 
$$
U:=U_{L,e_1}=\{y\in\mathbb Z:-L\le y\cdot e_1<L\}. 
$$
We define the Green operator on $L^\infty(U)$ by the formula
$$
G_U[f](x,\omega):= 
E_{x,\omega}\left[\sum_{n=0}^{T_U-1}f(X_n)\right],
$$
and then
\begin{equation*}
\hat\rho:=\sup\left\{\frac{1-\frac{1}{L}G_U[d\cdot e_1](x)}
  {1+\frac{1}{L}G_U[d\cdot e_1](x)}\colon x\cdot e_1=0, \sup_{2\le j\le d}
  |x\cdot e_j|<\frac{1}{4} M^3\right\}. 
\end{equation*}
It is simple to check that $\epsilon<\frac{1}{4d}$ implies $\sup_{x,\omega\in\Omega_\epsilon}E_{x,\omega}[T_U]<\infty$ and hence $G_U$ and $\hat\rho$ are well-defined. 
Let us next fix positive integers $h$ and $H$ satisfying $2h \le H \le \frac{1}{32}M^3$ (cf.~\cite[(1.9)]{Sz03}). Then define the stopping time 
$$
S:=\inf_{n\ge 0}\left\{|(X_n-X_0)\cdot e_1|\ge L\ {\rm or}\ \sup_{2\le
    j\le d}|(X_n-X_0)\cdot e_j|\ge h\right\}
$$
and the displacement variable
$$
\Delta(x,\omega):=E_{x,\omega}[X_S]-x.
$$
Finally define for $\gamma_1\in (0,1]$,
$$
p:=\inf_{j\ge 2}\mathbb P\left(\min_{z\in\tilde B_j}
\Delta(z,\omega)\cdot e_1\ge\gamma_1 L 
\right),
$$
where for $2\le j\le d$,
$$
\tilde B_j:=\left\{y\in B: |y\cdot e_j|<H\right\}.
$$
Now we can state a bound on $\E[\sqrt{\rho_B}]$ in terms of $\hat\rho$ and $p$. 
Note that our $M$ corresponds to $NL$ in~\cite{Sz03}.
\begin{lemma}
[Theorem~1.1 in~\cite{Sz03}]
  \label{coro}
Let $\bar{M}:=[M^3/(32H)]$ and assume that
  $$
  \delta^{-1}:=\exp\left\{-\frac{\gamma_1 M}{32L}\right\}
    +
    \frac{10 M}{\gamma_1 L}
    \exp\left\{-\frac{\gamma_1
        M}{32 L}\left(\frac{H L}{2hM}-\frac{4}{\gamma_1}
        \right)^2_+\right\}
    <1.
$$
Then
$$
\mathbb E[\sqrt{\rho_B}]
\le \kappa^{-2}\left(\frac{2\mathbb 
  E[\hat\rho]^{\frac{M}{2L}} }{(1-\E[\hat\rho]^{\frac{1}{2}})_+}
+
2d\kappa^{-\frac{M}{2}}\exp\left\{-\frac{\bar
    M}{2}\left(p-\frac{7M}{\bar M}\frac{\log\kappa^{-1}}{\log\delta}
    \right)^2_+
\right\}\right).
$$
\end{lemma}
\begin{remark}
In application, this lemma requires to check that $\E[\hat\rho]<1$ and $p$ is not too small. We will check these conditions in perturbative ways. First, since we assume $\E[d(0)\cdot e_1]>0$, if the fluctuation of $\omega$ is small, then it is reasonable to believe that $G_U[d\cdot e_1](x)>0$ for many points, which morally implies $\hat\rho<1$. Second, for the same reason, $\Delta(x,\omega)$ should be biased in the direction $e_1$ and we may (and will) choose $\gamma_1$ small so that $p$ is close to 1. 
\end{remark}

\subsection{High-level proof of Theorem~\ref{d4}}
\label{conclusion}
In this section, we will have to choose $\epsilon>0$ small as necessary. We write $\epsilon \ll 1$ instead of ``$\epsilon>0$ sufficiently small depending only on $d$ and $r$'', for simplicity. 

Let us write
\begin{equation}
  \label{lambda0def}
 \lambda_0:=\sqrt{r}\sigma_{2r}\epsilon^{2-\frac{1}{\sqrt{r}}}
\end{equation}
for the lower bound on $\lambda:=\E[d(0)\cdot e_1]$ assumed in Theorem~\ref{d4}. It is good to keep in mind that since $\sigma_{2r}$ is bounded, $\lim_{\epsilon\to 0}\lambda_0=0$. Note first that if $\sigma_{2r}\le\epsilon^2$, then 
 for $\epsilon\ll 1$,
\begin{equation*}
 \lambda\ge\sqrt{r}\sigma_{2r}^2\epsilon^{-\frac{1}{\sqrt{r}}} 
\ge 4d(1+9\epsilon) \sigma_2^2
\end{equation*}
since $\sigma_{2r} \ge (2d)^{-\frac{r-1}{r}}\sigma_2$.
By~\cite[Theorem 3]{RS18}, this implies Kalikow's condition and
the random walk is ballistic. We will hence assume that
\begin{equation}
 \sigma_{2r}>\epsilon^2. \label{nottoosmall}
\end{equation}
It remains to show that this and $\lambda\ge\lambda_0$ implies the condition~\eqref{effective}. 

We will use Lemma~\ref{coro} with the parameters
\begin{align}
&M:=\epsilon^{-\frac{1}{\sqrt{r}}}\lambda_0^{-1}
, \quad
L:=
\cst (d,r)\epsilon^{-1}, \label{lm}\\
&H:=M^2
,\quad h:=\epsilon^{-\frac{1}{2\sqrt{r}}}L^2
\quad \textrm{and} \quad 
\gamma_1:=\frac{\cst}{10}\lambda_0 L,\label{hHgchoice}
\end{align}
where the choice of $c_1(d,r)>0$ will be specified in the proof. 
By~\eqref{nottoosmall}, the requirement $M\ge \exp\{100+4d(\log \kappa)^2\}$ in~\eqref{effective} is satisfied when $\epsilon\ll 1$. 
Note also that the above choices of parameters satisfy $2h \le H \le \frac{1}{32}M^3$ in the previous section. 
Furthermore, we can compute 
$$
\frac{\gamma_1 M}{L}=\frac{c_2}{10}\epsilon^{-\frac{1}{\sqrt{r}}}\quad \text{and} \quad
\frac{HL}{2hM}-\frac{4}{\gamma_1}=\frac{1}{c_1\lambda_0}\left(
\frac{1}{2}\epsilon^{1-\frac{1}{2\sqrt{r}}}-\frac{40}{c_2} \epsilon\right).
$$
Using this, we find that there exists a constant $\cst(d,r)>0$ such that
for $\epsilon\ll 1$, 
$$
\delta^{-1}\le c_3(d,r)\exp\left\{-\epsilon^{-\frac{1}{2\sqrt{r}}}\right\}<1. 
$$
Hence from Lemma~\ref{coro}, it follows that for $\epsilon\ll 1$, 
\begin{equation}
\label{e1}
\mathbb E[\sqrt{\rho_B}]
\le {\kappa^{-2}} \frac{\mathbb 
  E[\hat\rho]^{\frac{M}{2L}} }{(1-\E[\hat\rho]^{\frac{1}{2}})_+}
+
2d\exp\left\{M\left[\frac{\log 4d}{2}-50\frac{\log 4d}{\log 2}
    \left(p-\frac{7}{100}\right)_+^2\right]\right\}.
\end{equation}
Therefore, 
if we prove that
\begin{equation}
  \label{prove}
\mathbb E[\hat\rho]\le 1-\frac{1}{10}d\lambda_0 L\quad{\rm and}\quad 
p\ge\frac{3}{4}, 
\end{equation}
then using~\eqref{e1}, we get 
\begin{equation}
  \label{mbound}
\mathbb E[\sqrt{\rho_B}]\le \kappa^{-2}\frac{80}{d\lambda_0 L}
\exp\left\{-\frac{d}{20}\lambda_0 M\right\}
+2d\exp\left\{-M\log 4d\right\}.
\end{equation}
Substituting~\eqref{lm} into~\eqref{mbound} and recalling~\eqref{sqrtrho}
, we can conclude that 
\begin{equation}
\label{pol}
P_0(X_{T_B}\notin\partial_+ B)\le \exp\left\{-\cst(d,r)\epsilon^{-\frac{1}{\sqrt{r}}}\right\},
\end{equation}
for some constant $c_4(d,r)>0$, whenever $\epsilon \ll 1$. This implies~\eqref{effective}. 

Let us verify (\ref{prove}) for $\epsilon \ll 1$. First, by the same arguments as in~\cite[(4.8)]{Sz03}, if we choose $c_1=\epsilon L<\frac{3}{4}$, then 
\begin{equation}
  \label{lower-bounds}
\sup_{\omega\in\Omega_\epsilon}\hat\rho \le 3.
\end{equation}
Thus in order to bound $\E[\hat\rho]$, we only need to prove that with high probability, $G_U[d\cdot e_1](x)$ is uniformly positive on 
\begin{equation*}
\mathcal{H}_M:= \left\{x\in \Z^d \colon x\cdot e_1=0, \sup_{2\le j\le d} |x\cdot e_j|<\frac{1}{4} M^3\right\}.
\end{equation*}
This follows from the following two propositions.  
\begin{proposition} 
\label{prop:mean}
Suppose $d\ge 3$ and~\eqref{eq:assum} holds. There exist constants $\cst(d), \cst(d)>0$ such that if
  $\epsilon\le \frac{1}{8d}$, $L\in [2, c_5(d)/\epsilon]\cap\N$ and 
\begin{equation}
\lambda\ge c_6(d)\sigma^2_2\left(\epsilon\log L+\frac{1}{L}\right),
\label{assum_mean}
\end{equation}
then
  \begin{equation}
    \label{expectation}
\mathbb E[G_U(d\cdot e_1)(0)]\ge\frac{2}{5} d\lambda L^2.
\end{equation}
\end{proposition}
\begin{proposition}
\label{prop:concentration}
Suppose $d\ge 4$ and $r\ge 2$ even. Then, there exists a constant 
 $\cst(d,r)>0$ such that if $L\le c_7(d,r)\epsilon^{-1}$, then
\begin{equation}
  \label{variance}
    \mathbb P\left(
|G_U[d\cdot e_1](0)-\mathbb E[G_U[d\cdot e_1](0)]|
\ge u\right) 
\le 
\begin{cases}
   \left(c_7(d,r) r\right)^{r} L
\left(\frac{\sigma_{2r}}{u}\right)^{2r} &\ {\rm for}\ d=4,\\
\left(c_7(d,r) r\right)^{r}
\left(\frac{\sigma_{2r}}{u}\right)^{2r} &\ {\rm for}\ d\ge 5.
\end{cases}
\end{equation}
for all $u\ge 0$. 
\end{proposition}
We postpone the justification of these propositions to the next section and complete the proof of~\eqref{prove} first. By using $\sigma_2\le \min\{(2d)^{\frac{r-1}{r}}\sigma_{2r}, \sqrt{2d}\epsilon\}$, one can verify that the assumption~\eqref{assum_mean} holds under~\eqref{eq:assum}. 
Thus we can use (\ref{expectation}) and (\ref{variance}) to deduce the following bound on the deviation probability for $\epsilon \ll 1$:
\begin{equation}
\label{dev}
\begin{split}
D:&=\mathbb P\left(\inf_{x\in \mathcal{H}_M
}
                    G_U[d\cdot e_1](x)
                    \le \frac{1}{5}d\lambda_0 L^2
                    \right)\\
&\stackrel{\mathclap{(\ref{expectation})}
}{\le}
\hspace{5pt}\mathbb P\left(
\inf_{x\in \mathcal{H}_M
}
                    \left(
G_U[d\cdot e_1](x)-\mathbb E[G_U[d\cdot e_1](x)]
                    \right)
                    \le-\frac{1}{5}d\lambda_0 L^2
                    \right)\\
                    &
                    \le
  2\left(\frac{M^3}{2}\right)^{d-1}
    \mathbb P\left(
    G_U[d\cdot e_1](x)-\mathbb E[G_U[d\cdot e_1](x)]
    \le-\frac{1}{5}d\lambda_0 L^2
    \right)\\
 &
  \stackrel{\mathclap{(\ref{variance}), (\ref{lambda0def})}
}{\le}
\hspace{15pt}
    2\left(\frac{M^3}{2}\right)^{d-1}\left(\frac{25 c_7(d)}{
        d^2}\right)^r L\epsilon^{2\sqrt{r}}\\
    &
    \overset{(\ref{nottoosmall})}{\le}
 \cst(d)^{r}\epsilon^{\sqrt{r}}\lambda_0 L,
\end{split}
\end{equation}
where in the last inequality we have used the condtion $r\ge 144d^2$ of Theorem~\ref{d4}, (\ref{nottoosmall}) and
the definition of $\lambda_0$ in   (\ref{lambda0def}).
From this and~\eqref{lower-bounds}, it follows that
\begin{equation*}
\begin{split}
\mathbb E[\hat\rho]
& \le\frac{1-\frac{1}{5}d\lambda_0
                    L}{1+\frac{1}{5}d\lambda_0 L}+
                    3D\\
  &
\le    1-\frac{1}{10}d\lambda_0 L.
\end{split}
\end{equation*}
Let us turn to prove that $p\ge \frac{3}{4}$. By the same argument as in~\cite[(4.12)]{Sz03} (see also~\cite[(2.10)]{Sz03}), one can prove that 
\begin{align*}
|\Delta(0,\omega)\cdot e_1-G_U(d\cdot e_1)(0)|
\le \frac{1}{5}d\lambda_0 L^2.
\end{align*}
Then by using (\ref{expectation}) and (\ref{variance}) again, we obtain
through a computation similar to the previous one that 
\begin{equation*}
\begin{split}
 p&\ge 1-\sup_{j\ge 2} M^{2d}\mathbb P\left(
 G_U[d\cdot e_1]-\mathbb E[G_U[d\cdot e_1](0)]\le-\frac{1}{10}
 d\lambda_0 L^2\right)\\
 &\ge 1-M^{2d} \cst(d)^r \epsilon^{\sqrt{r}}\lambda_0 L\\
 &\ge 1-\frac{1}{10}d\lambda_0 L\\
 &\ge\frac{3}{4},
\end{split}
\end{equation*}
where we have again used the condition $r\ge 144d^2$ in the third inequality.
This completes the proof of~\eqref{prove} and hence Theorem~\ref{d4}.  

\subsection{Lower bound on $G_U[d\cdot e_1](0)$}
\label{lowergreen}
In this section, we justify two propositions left unproved in the last section. 
First, Proposition~\ref{prop:mean} is nothing but~\cite[Proposition~7]{RS18}, which is an improvement of~\cite[Proposition~3.1]{Sz03}. 

So we only provide a proof of Proposition~\ref{variance}.
We will use an inequality of Boucheron, Bousquet, Lugossi and
Massart~\cite[Theorem 2]{BBLM05} (see also
the monograph~\cite[Theorem~15.5]{BLM13}), whose statement we reproduce here. 


\begin{theorem}
  \label{blm}
  For $N\ge 1$, let $X_1,\ldots,X_N$ be independent random
  variables taking values in a set $\mathcal X$, with
  laws $\mu_1,\ldots,\mu_N$ respectively,
  and $f:\mathcal X\to\mathbb R$ a measurable function. Let
$Z:=f(X_1,\ldots,X_N)$,  
  $$
  V_+:=\sum_{n=1}^{N}E'[(Z-Z'_n)_+^2]\quad \text{and}\quad
  V_-:=\sum_{n=1}^{N}E'[(Z-Z'_n)_-^2],
  $$
  where $X'_1,\ldots,X'_n$ are independent copies of $X_1,\ldots,X_n$ and
  $Z'_n$ is obtained from $Z_n$ by replacing $X_n$ by $X'_n$, while
  $E'$ denotes expectation with respect to the $(X'_1,\dotsc,X'_N)$ variables only. 
  Then, for any real $q\ge 2$, we have that 
  \begin{equation}
    \label{esq}
\|Z-E[Z]\|_q\le\sqrt{\frac{\sqrt{\text{\upshape{e}}}}{\sqrt{\text{\upshape{e}}}-1} q}\left(\|{V_+}\|_{q/2}^{1/2}+\|{V_-}\|_{q/2}^{1/2}\right),
  \end{equation}
where $\text{\upshape{e}}$ denotes Napier's constant.
\end{theorem}

\begin{proof}[Proof of Proposition~\ref{prop:concentration}] 
We will follow the proof of~\cite[Proposition 3.2]{Sz03}
  and of~\cite[Proposition 10]{RS18}. Let us enumerate the elements
  of $U$ as $U=\{x_n\colon n\in\mathbb N\}$.
    We will apply Theorem~\ref{blm} to the random variable
    $Z:=G_U[d\cdot e_1]$, which is a measurable function of the
    independent 
    random variables $\{X_n:=\omega(x_n)\colon n\in\N\}$. 
    Let us first note that by the Cauchy-Schwarz inequality, 
\begin{equation}
\begin{split}
 \E\left[V^r_+\right]&=
 \sum_{i_1\in\mathbb N}\cdots \sum_{i_r\in\mathbb N}\E[\E'[(Z-Z_{i_1})^2_+]\cdots 
                    \E'[(Z-Z_{i_r})^2_+]]\\
  &
    \le
    \sum_{i_1\in\mathbb N}\cdots \sum_{i_r\in\mathbb N}
    \E\left[\left(\E'[(Z-Z'_{i_1})_+^2\right)^{r}\right]^{1/r}
    \cdots \E\left[\left(\E'[(Z-Z'_{i_r})_+^2\right)^{r}\right]^{1/r}\\
  &
    =\left(\sum_{n\in\mathbb N} \E\left[\left(\E'[(Z-Z'_{n})_+^2\right)^{r}\right]^{1/r}\right)^r. 
 \end{split}
  \label{vz}
\end{equation}
We will now obtain an upper bound for the right-most expression
of~\eqref{vz}.
    For each $n$ and all environments
    $\omega,\omega'\in\Omega_\epsilon$
    with $\omega=\omega'$ off $x_n$, we define
    $$
    \Gamma_n(\omega,\omega'):=G_U[d\cdot e_1](0,\omega)
    -G_U[d\cdot e_1](0,\omega').
    $$
    Using Minkowski's integral inequality and Jensen's inequality, we see that
\begin{equation}
 \begin{split}
  \E\left[\left(\E'[(Z-Z'_{n})_+^2\right)^{r}\right]^{1/r}
&=
    \left(\int \left(\int 
        (\Gamma_n)_+^2(\omega,\omega')\mu(\dd\omega'(x_n)) 
                        \right)^r \mathbb P(\dd\omega)\right)^{1/r}\\
      & \le
        \int \left(\int 
        (\Gamma_n)_+^{2r}(\omega,\omega')\mathbb P (\dd\omega)
        \right)^{1/r}\mu(\dd\omega'(x_n))\\
&        \le
       \left( \iint 
        \Gamma_n^{2r}(\omega,\omega')\mathbb P(\dd\omega)\, 
                \mu(\dd\omega'(x_n)) \right)^{1/r}.
 \end{split}
      \label{dev1}
\end{equation}    
On the other hand, as in~\cite[(3.47)]{Sz03}, we have that for all
$\alpha\in (0,1)$,
\begin{equation}
    |\Gamma_n(\omega,\omega')|^2\le
    {\cst(d)}g_{0,U}(0,x_n)^{\frac{2}{2-\alpha}}\left(
      \sum_{e\in U}|\omega(x_n,e)-\omega'(x_n,e)|\right)^2
      \label{gamma}
\end{equation}
under the condition $L<\cst(d)\frac{1-\alpha}{2-\alpha}\epsilon^{-1}$ in~\cite[Proposition~2.3]{Sz03}, where $g_{0,U}$ denotes the Green function for the simple symmetric random walk killed upon exiting $U$. We set $\alpha:=1-\frac1r$. Hence the constants depending on $\alpha$ below actually depending on $r$. 
 Substituting~\eqref{gamma} into~\eqref{dev1}, we see that
    \begin{equation}
\begin{split}
   &\E\left[\left(\E'[(Z-Z'_{n})_+^2]\right)^{r}\right]^{1/r}\\
   & \quad\le
    \cst(d)g_{0,U}(0,x_n)^{\frac{2}{2-\alpha}}
    \left(\iint \left(\sum_{e\in
      U}|\tilde\delta(x_n,e)-
    \tilde\delta'(x_n,e)|\right)^{2r}\mu(\dd\omega(x_n))
        \mu(\dd\omega'(x_n))\right)^{1/r}
      \\
      &\quad
        \le
        2(4d)^2    c_{12}(d)g_{0,U}(0,x_n)^{\frac{2}{2-\alpha}}
        \sigma_{2r}^2.
\end{split}    
 \label{lastt}
\end{equation}
    Now, as explained in the proof of~\cite[Proposition
    3.2]{Sz03}, for $\alpha>\frac{4}{5}$, there exist
    constants $\cst(d,\alpha), \cst(d,\alpha)>0$ such that for $L\le
    c_{13}(d,\alpha)\epsilon^{-1}$, 
    $$
    \sum_{n\in\mathbb N}g_{0,U}(0,x_n)^{\frac{2}{2-\alpha}}\le
    \begin{cases}
    c_{14}(d,\alpha) L^{4\frac{1-\alpha}{2-\alpha}}
    & {\rm for} \ d=4,\\
    c_{14}(d,\alpha) & {\rm for} \ d\ge 5.
    \end{cases}
    $$
    Combining the above with (\ref{lastt}) and
     (\ref{vz}), we get that for all $\alpha\in (0,1]$, whenever $L\le 
        c_{13}(d,\alpha)\epsilon^{-1}$, 
\begin{equation}
    \label{vest}
    \|V_+\|_r \le
\begin{cases}
      \cst(d,\alpha) 
      L^{4\frac{1-\alpha}{2-\alpha}}
      \sigma_{2r}^2  &\ {\rm for}\ d=4,\\
      c_{15}(d,\alpha) 
      \sigma_{2r}^2  &\ {\rm for}\ d\ge 5,
  \end{cases}
\end{equation}
for some $c_{15}(d,\alpha)>0$. Recalling $\alpha=1-\frac{1}{4r}$ and applying Theorem \ref{blm} and estimate~\eqref{vest}, one can deduce~\eqref{variance}. 
\end{proof}

\section{Concluding remark}
Although the main theme of this work is to prove ballisticity under a \emph{stochastically} small perturbation, we still need a \emph{uniform} control on the perturbation. More precisely, it is~\eqref{lower-bounds}, Proposition~\ref{prop:mean} and \eqref{gamma} that require the uniform control. 
This is because they rely on the deterministic bounds on the Green functions proved in~\cite[Section~2]{Sz03}, which holds uniformly in $\omega\in\Omega_\epsilon$ for small $\epsilon$. 

It would be desirable to find suitable bounds on the Green functions associated with RWRE that hold with high probability, only assuming that the perturbation is stochastically small. 


\section*{Acknowledgments}
\noindent
This work was supported by the Research Institute for Mathematical Sciences, an International Joint Usage/Research Center located in Kyoto University. Ryoki Fukushima was supported by JSPS KAKENHI Grant Number 16K05200. 
A.F.~Ram\'\i rez was supported by Iniciativa Cient\'\i fica Milenio, Fondo Nacional de Desarrollo Cient\'\i fico y Tecnol\'ogico 1180259 and JSPS KAKENHI Grant Number JP17H01093. A.F.~Ram\'irez thanks the Technische Universit\"at M\"unchen, where part
of this work was done, for its support.

\end{document}